\documentclass[12pt]{amsart}
\usepackage[T1]{fontenc}
\usepackage{lmodern,microtype,amsmath,amssymb,booktabs,array}
\usepackage[hidelinks]{hyperref}
\calclayout
\newtheorem{theorem}{Theorem}[section]
\newtheorem{lemma}[theorem]{Lemma}

\newtheorem{corollary}[theorem]{Corollary}
\theoremstyle{remark}

\numberwithin{equation}{section}
\newcommand{\C}{\mathcal C}
\newcommand{\Zc}{\mathcal Z}
\newcommand{\one}{\mathbf 1}
\newcommand{\Irr}{\operatorname{Irr}}
\newcommand{\FPdim}{\operatorname{FPdim}}
\newcommand{\Rep}{\operatorname{Rep}}
\newcommand{\Vect}{\operatorname{Vec}}
\newcommand{\Hom}{\operatorname{Hom}}
\newcommand{\End}{\operatorname{End}}
\newcommand{\rank}{\operatorname{rank}}
\newcommand{\TY}{\operatorname{TY}}

\title[Fusion categories of rank six]{Grothendieck rings of strictly weakly integral fusion categories of rank six}
\author{Kai Wang}
\address{College of Mathematical Science, Yangzhou University, Yangzhou, China}
\email{wangkww@163.com}
\author{Jingcheng Dong}
\address{College of Mathematics and Statistics, Nanjing University of Information Science and Technology, Nanjing, China}
\email{jcdong@nuist.edu.cn}
\author{Libin Li}
\address{College of Mathematical Science, Yangzhou University, Yangzhou, China}
\email{lbli@yzu.edu.cn}

\date{}

\begin{document}
\begin{abstract}
We classify, up to based-ring isomorphism, the Grothendieck rings of strictly weakly integral fusion categories of rank six over $\mathbb C$. The dimension grading and the classification of integral Grothendieck rings of rank at most five reduce the problem to twelve candidate fusion rings. We exclude four of them using the twist trace and the induction pairing in the Drinfeld center. The remaining eight admit known categorifications. In particular, every such category has commutative Grothendieck ring and Frobenius--Perron dimension in $\{8,10,16,20\}$.
\end{abstract}
\maketitle

\section{Introduction}
The fusion rules of a fusion category record the decomposition of tensor products of simple objects. For a modular category, they can be studied through the modular data. Without a braiding, even a low-rank classification requires both an exhaustive restriction on the possible fusion rules and a decision as to which remaining rings admit categorifications. We carry out this program for strictly weakly integral fusion categories of rank six, using categorical constraints on their dimension gradings and Drinfeld centers.

A fusion category is strictly weakly integral if its Frobenius--Perron dimension is an integer and some simple object has nonintegral dimension. Bruillard, Galindo, Ng, Plavnik, Rowell and Wang classified weakly integral \emph{modular} categories through rank seven \cite{BGNPRW}. Liu, Palcoux and Ren classified the Grothendieck rings of fusion multiplicity one through rank six, without a braiding assumption \cite{LPR}. Neither hypothesis applies to a general category considered here. On the integral side, Alekseyev, Bruns, Dong and Palcoux classified the integral Grothendieck rings in ranks at most five \cite[Theorem~1.1]{ABDP}.

Earlier low-rank results include Ostrik's classifications in rank two and pivotal rank three \cite{O2,O3}, and the analysis of nontrivially graded, nonpointed self-dual spherical categories of rank four by Dong, Zhang and Dai \cite{DZD}. Dong, Palcoux and Plessis classified the Grothendieck rings of pivotal fusion categories of rank four \cite{DPP}. Wang and Dong studied faithfully graded categories of rank five~\cite{WD}.

\begin{theorem}\label{thm:main}
Up to based-ring isomorphism, the Grothendieck rings of strictly weakly integral fusion categories of rank six are precisely the eight rings $R_1,\ldots,R_6,R_9,R_{10}$ in Table~\ref{tab:main}.
\end{theorem}

\begin{table}[ht]
\caption{The eight Grothendieck rings. The last column records duality on the nonintegral simple objects.}
\label{tab:main}
\centering\small
\setlength{\tabcolsep}{5pt}
\begin{tabular}{@{}ccccc@{}}
\toprule
Ring & $\FPdim(\C)$ & Simple dimensions & $G(\C)$ & Duality\\
\midrule
$R_1$ & 8 & $1,1,1,1,\sqrt2,\sqrt2$ & $C_2^2$ & fixed\\
$R_2$ & 8 & $1,1,1,1,\sqrt2,\sqrt2$ & $C_2^2$ & exchanged\\
$R_3$ & 8 & $1,1,1,1,\sqrt2,\sqrt2$ & $C_4$ & fixed\\
$R_4$ & 8 & $1,1,1,1,\sqrt2,\sqrt2$ & $C_4$ & exchanged\\
$R_5$ & 20 & $1,1,2,2,\sqrt5,\sqrt5$ & $C_2$ & fixed\\
$R_6$ & 20 & $1,1,2,2,\sqrt5,\sqrt5$ & $C_2$ & exchanged\\
$R_9$ & 10 & $1,1,1,1,1,\sqrt5$ & $C_5$ & fixed\\
$R_{10}$ & 16 & $1,1,1,1,2,2\sqrt2$ & $C_2^2$ & fixed\\
\bottomrule
\end{tabular}
\end{table}

The labels in Table~\ref{tab:main} refer to the twelve candidate rings defined in Section~\ref{sec:candidates}. All eight surviving rings have known realizations: $R_1,\ldots,R_6,R_9$ occur in the multiplicity-one classification of Liu--Palcoux--Ren~\cite{LPR}, and $R_{10}$ is realized by a $C_2$-extension of the representation category of the Kac--Paljutkin Hopf algebra constructed by Choi--Lu--Sun~\cite[\S7]{CLS}. Precise identifications are given in the proofs below.

The contribution of this paper is the exhaustive reduction without a braiding or multiplicity-one assumption, together with the exclusion of the remaining candidates.

We first show that every strictly weakly integral fusion category $\C$ of rank six admits a faithful $C_2$-grading $\C=\C_0\oplus\C_1$, where $\C_0$ is integral and $(\rank(\C_0),\rank(\C_1))$ is either $(4,2)$ or $(5,1)$.
We then determine the possible Grothendieck rings arising from $C_2$-extensions of $\C_0$ in each case.
After applying the restrictions on rank-one grading components, twelve candidate fusion rings remain. In Section~\ref{sec:center}, formal codegrees determine the dimensions of the simple summands of the induced unit. The twist trace bounds their contributions to other induced objects, and the residual induction pairing must be a positive semidefinite Gram matrix. These two constraints exclude the four remaining candidates.

\section{Preliminaries}

Throughout, all fusion categories are over the field of complex numbers $\mathbb C$.

We use \(\operatorname{Irr}(\mathcal C)\) to denote the set of isomorphism classes of simple objects of \(\mathcal C\), and \(G(\mathcal C)\) to denote the group of isomorphism classes of invertible simple objects. We write $\rank(\C)=|\Irr(\C)|$. For an object $T$ of a fusion category, including its Drinfeld center, we set $d_T=\FPdim(T)$. The Grothendieck ring of \(\mathcal C\) is denoted by \(K_0(\mathcal C)\). When writing fusion rules, we identify simple objects with their classes in \(K_0(\mathcal C)\), and denote the class of the tensor unit by \(1\).

For objects \(U,V\in\mathcal C\), we write
$
[U,V]=\dim_{\mathbb C}\operatorname{Hom}_{\mathcal C}(U,V).
$
If \(U\) is simple, \([U,V]\) is the multiplicity of \(U\) in \(V\). Rigidity gives
\[
[U,VT]=[UT^*,V],\qquad [U,TV]=[T^*U,V].
\]
The Frobenius--Perron dimension of \(\mathcal C\) is
\(\sum_{T\in\operatorname{Irr}(\mathcal C)}d_T^2\).
We call \(\mathcal C\) integral if every simple object has integer Frobenius--Perron dimension, and weakly integral if \(\operatorname{FPdim}(\mathcal C)\) is an integer. Every weakly integral fusion category admits a canonical spherical structure whose categorical dimensions coincide with the Frobenius--Perron dimensions \cite[Propositions~8.23--8.24]{ENO}. We always use this structure and the induced one on the Drinfeld center, so categorical dimensions agree with Frobenius--Perron dimensions in both categories. We write $\C^{\mathrm{rev}}$ for the tensor category with reversed product $U\otimes^{\mathrm{rev}}V=V\otimes U$.

We denote by \(C_m\) the cyclic group of order \(m\), and by \(D_{2m}\) the dihedral group of order \(2m\). We set \(F_5=C_5\rtimes C_4\cong\operatorname{AGL}(1,5)\), where \(C_4\) acts faithfully on \(C_5\). Whenever \(C_7\rtimes C_3\) occurs, the action of \(C_3\) on \(C_7\) is understood to be nontrivial.

\subsection*{Gradings and rank-one components}
A grading of a fusion category $\C$ by a finite group $G$ is a decomposition into full abelian subcategories
\[
 \C=\bigoplus_{g\in G}\C_g,\qquad
 \one\in\C_e,\qquad
 \C_g\otimes\C_h\subseteq\C_{gh},\qquad
 (\C_g)^*=\C_{g^{-1}},
\]
where $e$ is the identity of $G$. The grading is faithful if $\C_g\ne0$ for every $g\in G$. In this case $\C$ is called a $G$-extension of $\C_e$, and all components have the same Frobenius--Perron dimension \cite[Proposition~8.20]{ENO}.

\begin{lemma}[{\cite[Theorem~3.10]{GN}}]\label{thm:dimgrading}
A weakly integral fusion category has a faithful grading $\C=\bigoplus_{x\in E}\C_x$ by an elementary abelian $2$-group $E$. There are distinct square-free positive integers $n_x$, with $n_0=1$, such that $d_T\in\mathbb Z\sqrt{n_x}$ for $T\in\Irr(\C_x)$.
\end{lemma}

\begin{lemma}\label{lem:one}
Let $\C=\C_0\oplus\C_1$ be a faithful $C_2$-grading with $\Irr(\C_1)=\{N\}$. The left action of $\C_0$ on $\C_1$ gives a fiber functor $\Phi$ on $\C_0$. Let $H$ be the corresponding semisimple Hopf algebra, so $(\C_0,\Phi)$ identifies with $\Rep(H)$ equipped with its forgetful functor. Then $\Rep(H)$ and $\Rep(H^*)$ have the same multiset of simple Frobenius--Perron dimensions. In particular, they have the same number of invertible simples. In the Grothendieck ring,
\begin{equation}\label{eq:rankone}
 N^*=N,\quad TN=NT=d_TN,\quad
 N^2=\sum_{T\in\Irr(\C_0)} d_T T,\quad d_N^2=\FPdim(\C_0),
\end{equation}
where the second identity holds for every $T\in\Irr(\C_0)$.
\end{lemma}
\begin{proof}
As a linear category, $\C_1\simeq\Vect$. Its category of
$\mathbb C$-linear endofunctors, with tensor product given
by composition, is tensor equivalent to $\Vect$ via
evaluation at $\mathbb C$. The left $\C_0$-action therefore
gives a tensor functor $\Phi:\C_0\to\Vect$. It is exact by semisimplicity and faithful: no simple object can be annihilated, since the unit is a summand of its product with its dual. Thus $\Phi$ is a fiber functor. Tannaka reconstruction gives a
finite-dimensional semisimple Hopf algebra $H$ such that
$(\C_0,\Phi)$ is tensor equivalent to $\Rep(H)$ equipped
with its forgetful functor \cite[\S\S5.2--5.3]{EGNO}.

The component $\C_1$ is an invertible $\C_0$-bimodule
category \cite[Theorem~6.1]{ENOH}. Write
$\operatorname{Fun}_{\C_0}(\C_1,\C_1)$ for the
category of left $\C_0$-module endofunctors, with
composition as tensor product. The right action identifies this category
with $\C_0^{\mathrm{rev}}$. Under the identification
of $(\C_0,\Phi)$ with $\Rep(H)$ equipped with its
forgetful functor, this category is tensor equivalent
to $\Rep(H^*)$ with the convention of
\cite[Definition~7.12.2 and Example~7.12.26]{EGNO}. Hence
\[
\C_0^{\mathrm{rev}}\simeq
\operatorname{Fun}_{\C_0}(\C_1,\C_1)\simeq\Rep(H^*).
\]
Reversing the tensor product preserves simple Frobenius--Perron dimensions. Since $\C_0\simeq\Rep(H)$, the two representation categories have the same multiset of such dimensions.

Duality preserves $\C_1$ and therefore fixes $N$. Each of $TN$ and $NT$ is a multiple of $N$, with multiplicity $d_T$ by dimensions. Reciprocity gives $[T,N^2]=[TN,N]=d_T$, proving \eqref{eq:rankone}.
\end{proof}

\begin{corollary}\label{cor:pq}
Let $\C=\C_0\oplus\C_1$ be a faithful $C_2$-grading with $\FPdim(\C_0)=pq$, where $p,q$ are distinct primes. If $\C_1$ contains exactly one simple object, then $\C_0$ is pointed.
\end{corollary}
\begin{proof}
	Let $H$ be the semisimple Hopf algebra reconstructed
	in Lemma~\ref{lem:one}. Since $\Phi$ is a fiber functor,
	$
	\dim H=\FPdim(\C_0)=pq.
	$
	By \cite{EGpq}, $H$ is a group algebra or a function
	algebra. If $\C_0$ were nonpointed, $H$ could be
	neither a function algebra nor the group algebra
	of an abelian group. Thus $H$ would be the group
	algebra of a nonabelian group, so $\Rep(H^*)$ would
	be pointed. This contradicts Lemma~\ref{lem:one}.
\end{proof}

When $\C_0\simeq\Rep(G)$, the fiber functor in
Lemma~\ref{lem:one} need not be the standard one.
By \cite[Proposition~5.14.4 and Corollary~7.12.24]{EGNO},
it corresponds to a twist of $\mathbb CG$, which can
be chosen in $\mathbb CL\otimes\mathbb CL$ for some
subgroup $L\leq G$ admitting a nondegenerate
$2$-cocycle $\psi$. Here nondegeneracy means that the twisted group algebra $\mathbb C_{\psi}L$ is simple, hence a full matrix algebra. Its dimension is $|L|$, so $|L|$ is a square.

\subsection*{The induction pairing and the twist trace}
Let $F:\Zc(\C)\to\C$ be the forgetful functor and $I$ its right adjoint. We use
\begin{equation}\label{eq:ind}
 F I(T)\simeq\bigoplus_{V\in\Irr(\C)}V\otimes T\otimes V^*,\qquad
 [I(S),I(T)]=[S,F I(T)] .
\end{equation}
The first identity is the standard formula for the induction
functor \cite[Proposition~5.4]{ENO}. The second follows from adjunction and symmetry of the Hom dimension in a semisimple category. In particular, $d_{I(T)}=\FPdim(\C)d_T$.

Suppose that $K_0(\C)$ is commutative. Set
\[
\Omega=\sum_{V\in\Irr(\C)}VV^*,
\]
and let $M_\Omega$ be its multiplication matrix in the basis $\Irr(\C)$: column $T$ is the coefficient vector of $\Omega T$. Thus \eqref{eq:ind} gives
\begin{equation}\label{eq:indmatrix}
(M_\Omega)_{S,T}=[S,\Omega T]=[I(S),I(T)].
\end{equation}
Every irreducible representation of the complexified fusion ring is one-dimensional. Consequently, the eigenvalues of $M_\Omega$ are the formal codegrees $f_i$, counted with multiplicity, by \cite[Remark~2.11]{O3}. Moreover, \cite[Theorem~2.13]{O3} gives
\begin{equation}\label{eq:codeg}
 I(\one)=\bigoplus_i A_i,\qquad d_{A_i}=\FPdim(\C)/f_i,\qquad [\one,F(A_i)]=1,
\end{equation}
where the $A_i$ are distinct simple objects of $\Zc(\C)$.

\begin{lemma}\label{lem:trace}
With the notation of \eqref{eq:codeg}, let $T\ne\one$ be simple and set
\begin{equation}\label{eq:trace}
 S_T=\sum_i[T,F(A_i)]d_{A_i}.
\end{equation}
Then $2S_T\le\FPdim(\C)d_T$.
\end{lemma}
\begin{proof}
By adjunction, $[T,F(A_i)]$ is the multiplicity of $A_i$ in $I(T)$. Theorem~2.5 of~\cite{O3} gives
\[
\theta_{A_i}=1,\qquad
\operatorname{Tr}(\theta_{I(T)})=0.
\]
Each remaining simple summand contributes its positive dimension times a root of unity \cite[Theorem~5.1]{Vafa}. The absolute value of their combined contribution is $S_T$ and is at most $d_{I(T)}-S_T$. Since $d_{I(T)}=\FPdim(\C)d_T$, the result follows.
\end{proof}

\begin{lemma}\label{lem:gram}
Let $A_1,\ldots,A_r$ be the distinct simple objects occurring in $I(\one)$, and let $\mathcal J$ be the set of all other simple objects of $\Zc(\C)$. For $T\in\Irr(\C)$, define
\[
v_T=([A_i,I(T)])_{i=1}^r,\qquad
B_T=([B,I(T)])_{B\in\mathcal J}.
\]
Equip each real coordinate space with the inner product
$\langle u,v\rangle=\sum_j u_jv_j$. Then
\begin{equation}\label{eq:residual}
G_{S,T}:=[I(S),I(T)]-\langle v_S,v_T\rangle
=\langle B_S,B_T\rangle.
\end{equation}
In particular, $G=(G_{S,T})_{S,T\in\Irr(\C)}$ is positive semidefinite and has nonnegative integer entries.
\end{lemma}
\begin{proof}
Decomposing $I(S)$ and $I(T)$ into simple objects gives \eqref{eq:residual}. All residual vectors use the same coordinate set $\mathcal J$, and their entries are nonnegative integers. A matrix of pairwise inner products of vectors is called their Gram matrix. For any family of real numbers $(x_T)$,
\[
\sum_{S,T}x_SG_{S,T}x_T
=\left\|\sum_T x_TB_T\right\|^2\ge0,
\]
which proves positive semidefiniteness.
\end{proof}

\section{Candidate fusion rings}\label{sec:candidates}

\begin{lemma}
	\label{lem:grading}
	Let $\C$ be a strictly weakly integral fusion category of rank six.
	Then its dimension grading is a faithful $C_2$-grading $\C=\C_0\oplus\C_1$.
\end{lemma}
\begin{proof}
	By Lemma~\ref{thm:dimgrading}, strict weak integrality implies that $E$ is nontrivial, while faithfulness gives $|E|\le6$. Since $E$ is an elementary abelian $2$-group, we have $|E|=2$ or $4$. Suppose $|E|=4$. Then the four nonempty components contain six simple objects, so at least two components have rank one. Their simple objects have equal dimensions because all components have the same Frobenius--Perron dimension. If $U\in\Irr(\C_x)$ and $V\in\Irr(\C_y)$ are these simple objects, with $x\ne y$, then $d_U=d_V$. By Lemma~\ref{thm:dimgrading}, this forces $n_x=n_y$, a contradiction.
\end{proof}

Write $\C=\C_0\oplus\C_1$ for the dimension grading. Thus $\C_0$ is integral and each simple object in $\C_1$ has dimension $a\sqrt n$ for some $a\in\mathbb Z_{>0}$, where $n>1$ is a common square-free integer.

\begin{theorem}\label{thm:smallneutral}
	Let $\C$ be a strictly weakly integral fusion category of
	rank six, and write $\C=\C_0\oplus\C_1$ for its dimension
	grading. Then
	\[
	(\operatorname{rank}(\C_0),\operatorname{rank}(\C_1))
	\in\{(4,2),(5,1)\}.
	\]
	\end{theorem}
\begin{proof}
If $\operatorname{rank}(\C_0)\le2$, then $\C_0$ is
pointed \cite{O2}, and hence
$\FPdim(\C_0)\le2$. Since the grading is faithful,
$\FPdim(\C_1)=\FPdim(\C_0)\le2$. On the other hand,
$\C_1$ has at least four simple objects, each of
dimension at least $\sqrt2$. Thus
$\FPdim(\C_1)\ge8$, a contradiction.

If $\rank(\C_0)=3$, the positive spherical structure on the integral category $\C_0$ permits us to apply \cite[Theorem~1.1]{O3}. Thus
$K_0(\C_0)$ is isomorphic to
$K_0(\Rep(C_3))$ or $K_0(\Rep(S_3))$. In the pointed case,
$\FPdim(\C_0)=3$, so $\FPdim(\C_1)=3$.
But $\C_1$ also has rank three, and each of its
simple objects has dimension at least $\sqrt2$.
Consequently, $\FPdim(\C_1)\ge6$, a contradiction.

In the $S_3$ case $\Irr(\C_0)=\{1,X,Y\}$, with dimensions $1,1,2$. The squared dimensions of the three simple objects in $\C_1$ sum to six, so each has dimension $\sqrt2$. For every $U\in\Irr(\C_1)$ we have $UU^*=1+X$. Hence $XU=U$ by reciprocity. Applying this to $U^*$ and dualizing gives $UX=U$ as well. Choose distinct $U,V\in\Irr(\C_1)$. Neither $1$ nor $X$ occurs in $U^*V$, since $[1,U^*V]=[U,V]=0$ and $[X,U^*V]=[UX,V]=0$. Its dimension is two, so $U^*V=Y$. The rigidity adjunctions give
\[
\Hom(U^*V,U^*V)\simeq\Hom(UU^*V,V)
\simeq\Hom(UU^*,VV^*).
\]
The last space has dimension $\dim\End(1+X)=2$, contradicting the simplicity of $Y$.
\end{proof}

\subsection{The case \texorpdfstring{$(\rank(\C_0),\rank(\C_1))=(4,2)$}{(rank(C0), rank(C1)) = (4,2)}}

By \cite[Theorem~1.1]{ABDP},
\[
 K_0(\C_0)\cong K_0(\Rep(G)),
 \qquad G\in\{C_2^2,C_4,D_{10},A_4\}.
\]
We determine the possible products involving the two simple objects in $\C_1$.

\begin{theorem}\label{thm:r14}
	Suppose that
	$K_0(\C_0)\cong K_0(\Rep(C_2^2))$ or
	$K_0(\Rep(C_4))$. Write $\Irr(\C_0)=\{1,X_1,X_2,X_3\}$ and
	$\Irr(\C_1)=\{M_1,M_2\}$. Then $K_0(\C)$ is one of the following four
	commutative rings $R_1,R_2,R_3,R_4$.
\begin{enumerate}
\item In $R_1$ and $R_2$, the common rules are
\begin{gather*}
	X_1^2=X_2^2=X_3^2=1,\quad
	X_1X_2=X_3,\quad X_1X_3=X_2,\quad X_2X_3=X_1,\\
	X_1M_1=M_1,\quad X_1M_2=M_2,\quad
	X_2M_1=X_3M_1=M_2,\quad X_2M_2=X_3M_2=M_1.
\end{gather*}
The products of $M_1,M_2$ are
\[
\begin{array}{c|cc}
 & M_1^2=M_2^2 & M_1M_2\\ \hline
 R_1 & 1+X_1 & X_2+X_3\\
 R_2 & X_2+X_3 & 1+X_1 .
\end{array}
\]
\item In $R_3$ and $R_4$, the common rules are
\begin{gather*}
	X_1^2=X_3^2=X_2,\quad X_2^2=1,\quad
	X_1X_2=X_3,\quad X_1X_3=1,\quad X_2X_3=X_1,\\
	X_2M_1=M_1,\quad X_2M_2=M_2,\quad
	X_1M_1=X_3M_1=M_2,\quad X_1M_2=X_3M_2=M_1.
\end{gather*}
The products of $M_1,M_2$ are
\[
\begin{array}{c|cc}
 & M_1^2=M_2^2 & M_1M_2\\ \hline
 R_3 & 1+X_2 & X_1+X_3\\
 R_4 & X_1+X_3 & 1+X_2 .
\end{array}
\]
\end{enumerate}
The ring $R_1$ is realized by $\Vect_{C_2}\boxtimes\mathrm{SU}(2)_2$, and $R_2,R_3,R_4$ are realized by zestings of $\Vect_{C_2}\boxtimes\mathrm{SU}(2)_2$. For the definition of zesting, see \cite{Zest}.
\end{theorem}
\begin{proof}
	Let $\Irr(\C_0)=\{\one,X_1,X_2,X_3\}$ and
	$\Irr(\C_1)=\{M_1,M_2\}$. Since $\C_0$ is pointed,
	the tensor product of any two simple objects in $\C_0$
	is simple. For all $i,j\in\{1,2\}$, the product
	$M_i\otimes M_j$ lies in $\C_0$, so its simple
	constituents are all invertible. For an invertible object $g$, Frobenius reciprocity gives
\[
[g,M_iM_j]=[gM_j^*,M_i]\le1.
\]
Thus each constituent occurs with multiplicity at most one. On the other
	hand, $X_i\otimes M_j$ and $M_j\otimes X_i$ are simple
	objects for all $i\in\{1,2,3\}$ and $j\in\{1,2\}$.
	It follows that $\C$ is a fusion category of multiplicity
	one. Since $\FPdim(\C)=2\FPdim(\C_0)=8$, the four
	rings and their realizations follow from
\cite[\S\S3.6, 4.3, 4.6]{LPR}. More precisely, $R_1,R_2,R_3,R_4$ correspond respectively to entries Nos.~3, 6, 5, 4 in the rank-six list of \cite[\S3.6]{LPR}. If $b_0,\ldots,b_5$ denote the basis in the order of the matrices there, the identification for $R_1,R_2$ is
\[
(1,X_1,X_2,X_3,M_1,M_2)=(b_0,b_1,b_2,b_3,b_4,b_5),
\]
whereas for $R_3,R_4$ it is
\[
(1,X_1,X_2,X_3,M_1,M_2)=(b_0,b_1,b_3,b_2,b_4,b_5).
\]
\end{proof}

\begin{theorem}\label{thm:r56}
Suppose that $K_0(\C_0)\cong K_0(\Rep(D_{10}))$. Write $\Irr(\C_0)=\{1,X,Y,Z\}$ and $\Irr(\C_1)=\{M,N\}$. Then $K_0(\C)$ is one of the two commutative rings $R_5,R_6$.
In both rings
\begin{align*}
	&X^2=1,\quad XY=Y,\quad XZ=Z,\quad YZ=Y+Z,\quad
	XM=N,\quad XN=M,\\
	&Y^2=1+X+Z,\quad Z^2=1+X+Y,\quad
	YM=YN=ZM=ZN=M+N.
\end{align*}
The products of $M$ and $N$ are given by
\[
\begin{array}{c|cc}
 & M^2=N^2 & MN\\ \hline
 R_5 & 1+Y+Z & X+Y+Z\\
 R_6 & X+Y+Z & 1+Y+Z .
\end{array}
\]
\end{theorem}
\begin{proof}
If $d_M=a\sqrt n$ and $d_N=b\sqrt n$, then $n(a^2+b^2)=10$. Up to order, the dimension pair is $(\sqrt2,2\sqrt2)$ or $(\sqrt5,\sqrt5)$. In the first case dimensions force both $M$ and $N$ to be self-dual and $XM=M$. Hence $M^2=1+X$. By reciprocity neither $YM$ nor $ZM$ contains $M$. Dimensions then give $YM=ZM=N$, so associativity gives $NM=Y(1+X)=2Y$ and also $NM=Z(1+X)=2Z$, a contradiction.

In the second case, if $XM=M$, both $1$ and $X$ occur in $MM^*$ with multiplicity one. The remaining dimension is three and would have to come from simples of dimension two. Thus $X$ interchanges $M,N$. Since $X(YM)=YM$, the multiplicities of $M,N$ in $YM$ are equal, and their sum is two. The same reasoning gives the other three products with $Y,Z$. Dualizing yields the right actions. Reciprocity now gives $MM^*=NN^*=1+Y+Z$ and $MN^*=NM^*=X+Y+Z$. The two possible dualities give the two rings. The rings $R_5,R_6$ are respectively entries Nos.~20 and 21 in \cite[\S3.6]{LPR}, with the displayed basis order $1,X,Y,Z,M,N$. They are realized by $\mathrm{SO}(5)_2$ and the zesting described in \cite[\S4.5 and Theorem~4.7]{LPR}.
\end{proof}

\begin{theorem}\label{thm:r78}
	Suppose that $K_0(\C_0)\cong K_0(\Rep(A_4))$. Write $1,X,Y=X^2$
	for the invertibles with $X^3=1$, and let $Z$ be the
	simple object of dimension three in $\C_0$. Write
	$\Irr(\C_1)=\{M,N\}$. Then $K_0(\C)$ is one of the
	two commutative rings $R_7,R_8$. In both rings
\begin{gather*}
	XZ=YZ=Z,\qquad Z^2=1+X+Y+2Z,\\
	XM=YM=M,\quad XN=YN=N,\quad
	ZM=M+2N,\quad ZN=2M+N.
\end{gather*}
	The products of $M$ and $N$ are given by
	\[
	\begin{array}{c|cc}
		& M^2=N^2 & MN\\ \hline
		R_7 & 1+X+Y+Z & 2Z\\
		R_8 & 2Z & 1+X+Y+Z
	\end{array}
	\]
\end{theorem}
\begin{proof}
Write $d_M=a\sqrt n$ and $d_N=b\sqrt n$. Then $n(a^2+b^2)=12$, which forces $a=b=1$ and $n=6$. The order-three group of invertibles acts trivially on $\{M,N\}$ on either side. Thus $MM^*$ contains $1,X,Y$ once each. Its remaining dimension is three, so $MM^*=1+X+Y+Z$. Reciprocity gives $[M,ZM]=1$, and dimensions force $ZM=M+2N$. The calculation for $N$ gives $NN^*=1+X+Y+Z$ and $ZN=2M+N$. Also $MN^*$ and $NM^*$ contain no invertible and hence are $2Z$. Duality gives the right action of $Z$ and the two alternatives for the products of $M,N$.
\end{proof}

\subsection{\texorpdfstring{The case $(\rank(\C_0),\rank(\C_1))=(5,1)$}{The case (rank(C0), rank(C1)) = (5,1)}}

By \cite[Theorem~1.1]{ABDP}, $K_0(\C_0)$ is isomorphic to one of the following: $K_0(\Rep(G))$ for
$
G\in\{C_5,D_8,D_{14},F_5,C_7\rtimes C_3,S_4,A_5\},
$
the Tambara--Yamagami ring on $C_4$, or the Grothendieck ring of the isotype variation (but non-zesting) of $\Rep(S_4)$ described in \cite[\S4.4]{LPR}. Table~\ref{tab:neutral5} summarizes the reduction; its last column refers to the proof of Theorem~\ref{thm:r912}, except for the two later center exclusions.

\begin{table}[ht]
\caption{The rank-five integral $K_0(\C_0)$.}
\label{tab:neutral5}
\centering\small
\begin{tabular}{@{}lcl@{}}
\toprule
$K_0(\C_0)$ & $\FPdim(\C_0)$ & Outcome\\
\midrule
$K_0(\Rep(C_5))$ & 5 & $R_9$\\
$K_0(\Rep(D_8))$ & 8 & $R_{10}$\\
$\TY(C_4)$ ring & 8 & excluded: dimension-eight Hopf algebras\\
$K_0(\Rep(D_{14}))$ & 14 & excluded: Corollary~\ref{cor:pq}\\
$K_0(\Rep(F_5))$ & 20 & excluded: twist-support obstruction\\
$K_0(\Rep(C_7\rtimes C_3))$ & 21 & excluded: Corollary~\ref{cor:pq}\\
$S_4$ isotype variation & 24 & $R_{11}$; excluded in Theorem~\ref{thm:no1112}\\
$K_0(\Rep(S_4))$ & 24 & $R_{12}$; excluded in Theorem~\ref{thm:no1112}\\
$K_0(\Rep(A_5))$ & 60 & excluded: twist-support obstruction\\
\bottomrule
\end{tabular}
\end{table}

\begin{theorem}\label{thm:r912}
If $\rank(\C_0)=5$, write $\Irr(\C_1)=\{N\}$. Then $K_0(\C)$ is one of the four rings $R_9,R_{10},R_{11},R_{12}$ below.
\begin{enumerate}
\item $R_9$ is the Tambara--Yamagami ring on $C_5$.
\item For $R_{10}$, we have $K_0(\C_0)\cong K_0(\Rep(D_8))$, and the products involving $N$ are given by \eqref{eq:rankone}.
\item For $R_{11}$ and $R_{12}$, write $\Irr(\C_0)=\{1,X,Y,Z,M\}$. The products in $K_0(\C_0)$ are commutative and satisfy
\begin{gather*}
	X^2=1,\quad XY=Y,\quad XZ=M,\quad XM=Z,\\
	Y^2=1+X+Y,\quad YZ=YM=Z+M,\\
	\begin{array}{c|cc}
		& Z^2=M^2 & ZM\\ \hline
		R_{11} & X+Y+Z+M & 1+Y+Z+M\\
		R_{12} & 1+Y+Z+M & X+Y+Z+M .
	\end{array}
\end{gather*}
In both rings $TN=NT=d_T N$ for each $T\in\Irr(\C_0)$, and
\[
N^2=1+X+2Y+3Z+3M .
\]
\end{enumerate}

\end{theorem}
\begin{proof}
Let $H$ be reconstructed from the fiber functor supplied by $N$ in Lemma~\ref{lem:one}.

If $K_0(\C_0)\cong K_0(\Rep(D_{14}))$ or
$K_0(\Rep(C_7\rtimes C_3))$, then
$\FPdim(\C_0)=14$ or $21$, respectively. In both cases
$\C_0$ is nonpointed, contradicting
Corollary~\ref{cor:pq}.

If $K_0(\C_0)\cong K_0(\Rep(F_5))$, then $K_0(\Rep(F_5))$ is the near-group ring $(C_4,3)$. By \cite[Theorem~1.3 and \S3.3]{Siehler} it has a unique tensor categorification over $\mathbb C$, so $\C_0\simeq\Rep(F_5)$. For a fiber functor on $\Rep(F_5)$, the subgroup supporting the twist has square order dividing $20$. The only possibility besides $1$ has order four. Every subgroup of order four is a cyclic Sylow $2$-subgroup, and $H^2(C_4,\mathbb C^\times)=0$. It therefore admits no nondegenerate two-cocycle. The twist is trivial up to gauge, so $H$ is the group algebra of $F_5$ and $\Rep(H^*)$ is pointed. This contradicts Lemma~\ref{lem:one}.

If $K_0(\C_0)\cong K_0(\Rep(A_5))$, then it has no nontrivial proper fusion subring, since a fusion subcategory of $\Rep(G)$ comes from a quotient of $G$, and $A_5$ is simple. Thus $\C_0$ is simple, and \cite[Theorem~9.12]{ENOW} gives a tensor equivalence $\C_0\simeq\Rep(A_5)$. The chosen fiber functor reconstructs a twist $(\mathbb C A_5)^J$. The trivial twist gives a pointed dual category of rank $60$. For a nontrivial twist the support subgroup must have order four, the only nontrivial square dividing $60$, and such a subgroup is a Klein four-group $L$. We may take $J\in\mathbb C L\otimes\mathbb C L$ \cite[Corollary~7.12.24]{EGNO}. This is also the twist appearing in \cite[Corollary~9.14]{ENOW}. Since $L$ is abelian, $\mathbb C L\otimes\mathbb C L$ is commutative. Hence, for each $h\in L$, the element $h\otimes h$ commutes with $J$ and consequently $\Delta^J(h)=h\otimes h$. Thus $H$ has at least four group-like elements, giving at least four one-dimensional $H^*$-modules. But $\Rep(A_5)$ has only one invertible simple. Neither twist can satisfy Lemma~\ref{lem:one}.

If $K_0(\C_0)\cong K_0(\Rep(D_8))$, then
\eqref{eq:rankone} gives $R_{10}$.
Since $K_0(\Rep(D_8))\cong K_0(\Rep(H_8))$, where
$H_8$ is the Kac--Paljutkin Hopf algebra, this ring is
realized by a $C_2$-extension of $\Rep(H_8)$ constructed
by Choi, Lu and Sun \cite[\S7, Tables~3--4]{CLS}. Their additional simple object $D$ corresponds to our $N$; their simple object of dimension two lies in the neutral component. The displayed associator data in that reference supply categorifications of the rank-one extension with products \eqref{eq:rankone}.

If $K_0(\C_0)$ is the Tambara--Yamagami ring on $C_4$,
then $H$ is an eight-dimensional semisimple Hopf algebra
whose representation ring has four invertible simples
forming $C_4$ and one two-dimensional simple. This is
impossible by Masuoka's classification \cite{Mas}:
every nonpointed representation ring of an
eight-dimensional semisimple Hopf algebra has four
invertible simples forming $C_2^2$ and one
two-dimensional simple, while all remaining cases
are pointed.

\smallskip\noindent\emph{The remaining rings.}
If $K_0(\C_0)\cong K_0(\Rep(C_5))$, formula~\eqref{eq:rankone} gives $R_9$. It is realized by a Tambara--Yamagami category on $C_5$ \cite{TY}, and is entry No.~7 in the rank-six list of \cite[\S3.6]{LPR}. If $K_0(\C_0)\cong K_0(\Rep(S_4))$, then
\eqref{eq:rankone} gives $R_{12}$. If $K_0(\C_0)$
is isomorphic to the Grothendieck ring of the
isotype variation (but non-zesting) of $\Rep(S_4)$
described in \cite[\S4.4]{LPR}, then
\eqref{eq:rankone} gives $R_{11}$.
\end{proof}

\section{Obstructions from the Drinfeld center}\label{sec:center}
It remains to exclude $R_7,R_8,R_{11},R_{12}$. All four rings are commutative. The two rings in each pair have different dualities and some different products. Nevertheless, direct multiplication using Theorems~\ref{thm:r78} and~\ref{thm:r912} shows that both the coefficient vector of $\Omega$ and its multiplication matrix agree within each pair, in the corresponding ordered bases below. Thus one induction calculation suffices for each pair.

\begin{theorem}\label{thm:no78}
Neither $R_7$ nor $R_8$ is the Grothendieck ring of a fusion category.
\end{theorem}
\begin{proof}
Suppose that one of the two rings is realized by $\C$. Its simple dimensions, in the order $1,X,Y,Z,M,N$, are $1,1,1,3,\sqrt6,\sqrt6$, so $\FPdim(\C)=24$. In both cases $\Omega=6+3X+3Y+4Z$, with multiplication matrix
\[
 M_\Omega=
 \begin{pmatrix}
 6&3&3&4&0&0\\
 3&6&3&4&0&0\\
 3&3&6&4&0&0\\
 4&4&4&20&0&0\\
 0&0&0&0&16&8\\
 0&0&0&0&8&16
 \end{pmatrix}.
\]
Its characteristic polynomial is $(t-24)^2(t-8)^2(t-3)^2$. Formula~\eqref{eq:codeg} gives two summands of $I(\one)$ of dimension one, two of dimension three and two of dimension eight. Each has exactly one copy of $\one$ after forgetting. None of the forgetful images contains $M$ or $N$, since all six images have integral dimension.

The two one-dimensional summands forget to $\one$.
The forgetful image of each three-dimensional summand therefore contains
the unit and two further simple constituents, each
equal to $X$ or $Y$, counted with multiplicity.
The four copies of $Z$ in $FI(\one)$ must all occur
in the two eight-dimensional summands, since the
smaller summands cannot contain $Z$. Each
eight-dimensional summand can contain at most two
copies of $Z$. Hence each contains exactly two copies
of $Z$ and, in addition to the unit, exactly one of
$X,Y$. Summing the contributions to $S_X+S_Y$ gives
\[
S_X+S_Y=2\cdot2\cdot3+2\cdot1\cdot8=28.
\]
Since $X,Y$ are nonunit invertibles, Lemma~\ref{lem:trace}
gives $S_X+S_Y\le12+12=24$, a contradiction.
\end{proof}

\begin{theorem}\label{thm:no1112}
Neither $R_{11}$ nor $R_{12}$ is the Grothendieck ring of a fusion category.
\end{theorem}

\begin{proof}
	Suppose such a category exists. Order its simples as
	$1,X,Y,Z,M,N$. Their dimensions are
	$1,1,2,3,3,2\sqrt6$, so $\FPdim(\C)=48$. In both rings
	$\Omega=6+2X+5Y+5Z+5M$, and the fusion rules of
	Theorem~\ref{thm:r912} give
	\begin{equation}\label{eq:48mat}
		M_\Omega=
		\begin{pmatrix}
			6&2&5&5&5&0\\
			2&6&5&5&5&0\\
			5&5&13&10&10&0\\
			5&5&10&21&17&0\\
			5&5&10&17&21&0\\
			0&0&0&0&0&48
		\end{pmatrix}.
	\end{equation}
	Its characteristic polynomial is
	$(t-48)^2(t-8)(t-4)^2(t-3)$. By \eqref{eq:codeg},
	the simple summands $A_1,\ldots,A_6$ of $I(\one)$
	have dimensions $1,1,6,12,12,16$, respectively, and
	each $F(A_i)$ contains $\one$ exactly once. Hence
	$F(A_1)=F(A_2)=\one$. Since $d_N=2\sqrt6$, none of
	the $F(A_i)$ contains $N$. For $3\le i\le6$, write
	$F(A_i)=\one+a_iX+b_iY+c_iZ+e_iM$, and put
	$s_i=c_i+e_i$. Set
	$a=(a_3,a_4,a_5,a_6)$, $b=(b_3,b_4,b_5,b_6)$ and
	$s=(s_3,s_4,s_5,s_6)$.
	
	Since $F$ preserves FP dimensions and
	$\sum_{i=3}^6F(A_i)=4\one+2X+5Y+5Z+5M$, we obtain
	\begin{equation}\label{eq:48rows}
		\begin{gathered}
			a+2b+3s=(5,11,11,15),\\
			\sum_{i=3}^6a_i=2,\qquad
			\sum_{i=3}^6b_i=5,\qquad
			\sum_{i=3}^6s_i=10.
		\end{gathered}
	\end{equation}
	
	By Lemma~\ref{lem:trace} and adjunction,
	\[
	48\ge S_Y
	=6b_3+12b_4+12b_5+16b_6
	=60-6b_3+4b_6.
	\]
	Since $a_3+2b_3+3s_3=5$, we have $b_3\le2$.
	Together with $b_6\ge0$, this forces $b_3=2$ and
	$b_6=0$. Consequently, $a_3=1$, $s_3=0$.
	Similarly, $a_6+3s_6=15$ and $0\le a_6\le1$ give
	$a_6=0$, $s_6=5$. Thus $a_4+a_5=1$. Interchanging
	$A_4,A_5$ if necessary, assume $a_4=1$, $a_5=0$.
We have
$
1+2b_4+3s_4=11,
2b_5+3s_5=11.
$
Thus $b_4\equiv2$ and $b_5\equiv1\pmod3$.
Since $b_4+b_5=3$, we get $b_4=2$ and $b_5=1$.
Substituting back gives $s_4=2$ and $s_5=3$. Therefore
	$b=(2,2,1,0)$ and $s=(0,2,3,5)$.
	
	Now set $W=B_Z+B_M$ as in Lemma~\ref{lem:gram}.
	By adjunction, the multiplicity vectors of the
	$A_i$ in $I(Y)$ and $I(Z)+I(M)$, after omitting
	the first two zero coordinates, are $b$ and $s$,
	respectively. By \eqref{eq:indmatrix} and Lemma~\ref{lem:gram},
\[
\langle B_Y,B_Y\rangle=13-\langle b,b\rangle=4,\quad
\langle B_Y,W\rangle=20-\langle b,s\rangle=13,\quad
\langle W,W\rangle=76-\langle s,s\rangle=38.
\]
	The Gram matrix of $B_Y,W$ is therefore
\[
\begin{pmatrix}4&13\\13&38\end{pmatrix},
\]
whose determinant is $4\cdot38-13^2=-17$. This contradicts positive semidefiniteness in Lemma~\ref{lem:gram}; equivalently, it violates the Cauchy--Schwarz inequality.
\end{proof}

\begin{proof}[Proof of Theorem~\ref{thm:main}]
Lemma~\ref{lem:grading} and Theorem~\ref{thm:smallneutral} reduce the dimension grading to the two rank pairs treated in Section~\ref{sec:candidates}. Theorems~\ref{thm:r14}--\ref{thm:r912} then yield twelve candidate fusion rings. Theorems~\ref{thm:no78} and~\ref{thm:no1112} exclude $R_7,R_8,R_{11},R_{12}$, and the remaining eight have the realizations specified in Section~\ref{sec:candidates}.

It remains to distinguish the surviving rings. Their global Frobenius--Perron dimensions separate the classes of dimensions $8,10,16,20$. Among the dimension-eight rings, the group of invertible basis elements is $C_2^2$ for $R_1,R_2$ and $C_4$ for $R_3,R_4$. In each pair, duality fixes the two nonintegral basis elements in the first ring and exchanges them in the second. The same distinction separates $R_5$ and $R_6$. These are based-ring invariants, so the eight rings are pairwise nonisomorphic.
\end{proof}


\end{document}